\documentclass[english,a4paper,12pt]{amsart}
\usepackage[a4paper,lmargin=2cm,rmargin=2cm,tmargin=4cm,bmargin=4cm]{geometry}
\usepackage[centertags]{amsmath}
\usepackage{amsfonts}
\usepackage{amssymb}
\usepackage{amsthm}
\usepackage{stmaryrd} 
\usepackage{graphicx}
\usepackage[colorlinks]{hyperref}

\newcommand{\Natural}{\mathbb N}

\newcommand{\Real}{\mathbb R}

\newcommand{\abs}[1]{\left\vert#1\right\vert}
\newcommand{\set}[1]{\left\{#1\right\}}

\newcommand{\cardinality}[1]{\abs{#1}}

\newcommand{\dist}{\mathop{\mathrm{dist}}\nolimits}

\newcommand{\norm}[1]{\left\Vert#1\right\Vert}
\newcommand{\duality}[1]{\left\langle#1\right\rangle}
\newcommand{\clco}{\mathop{\overline{\mathrm{co}}}\nolimits}
\newcommand{\closedball}[1]{B_{#1}}

\newcommand{\sphere}[1]{S_{#1}}

\newcommand{\LipEmb}[1]{\displaystyle \mathop{\hookrightarrow}_{#1}}
\newcommand{\basepoint}{{\mathbf 0}}

\newcommand{\nsub}[2]{{#1 \choose #2}}
\newcommand{\Free}{{\mathcal F}}

\theoremstyle{plain}
\newtheorem{thm}{Theorem}
\newtheorem*{thm*}{Theorem}

\newtheorem{lem}[thm]{Lemma}
\newtheorem{prop}[thm]{Proposition}

\theoremstyle{definition}

\newtheorem{rem}[thm]{Remark}

\begin{document}
\title{Linear properties of Banach spaces and low distortion embeddings of metric graphs}
\author{Anton\'\i n Proch\'azka}
\address{Universit\'e Franche-Comt\'e\\
Laboratoire de Math\'ematiques UMR 6623\\
16 route de Gray\\
25030 Besan\c con Cedex\\
France}
\email{antonin.prochazka@univ-fcomte.fr}
\dedicatory{Dedicated to the memory of Luis S\'anchez Gonz\'alez}
\begin{abstract}
We characterize non-reflexive Banach spaces by a low-distortion (resp. isometric) embeddability of a certain metric graph up to a renorming.
Also we study non-linear sufficient conditions for $\ell_1^n$ being $(1+\varepsilon)$-isomorphic to a subspace of a Banach space $X$.
\end{abstract}
\maketitle
\section{Introduction}
In \cite{ProSan}, L. S\'anchez and the author have shown the following characterization.
\begin{thm}\label{t:ProSan}
There is a bounded countable metric graph $M_{\ell_1}$ with the following properties.\\
a) If $M_{\ell_1}$ Lipschitz-embeds into a Banach space $X$ with distortion $D<2$ (denoted $M \LipEmb{D} X$), then $X$ contains an isomorphic copy of $\ell_1$. 

b) Conversely, if $X$ contains an isomorphic copy of $\ell_1$ then there is an equivalent norm $\abs{\cdot}$ on $X$ such that $M_{\ell_1}$ embeds isometrically into $(X,\abs{\cdot})$.
\end{thm}
In this article we modify the methods used to prove this theorem to obtain some further results.
Namely, the aim of Section~\ref{s:reflexive} is to present a similar bounded countable metric graph $M_{R}$ which satisfies the above theorem with the property ``$X$ contains an isomorphic copy of $\ell_1$'' replaced by the property ``$X$ is non-reflexive'' (Theorem~\ref{t:reflexive}).
In Section~\ref{s:Stability} we briefly discuss the importance of the renorming in these theorems and answer some quantitative questions left open in~\cite{ProSan}.

In Section~\ref{s:local} we establish a local version of Theorem~\ref{t:ProSan} a), i.e. a theorem where $M_{\ell_1}$ is replaced by a finite metric space and $\ell_1$ is replaced by a finite dimensional $\ell_1^n$.
In fact, using an ultraproduct argument one can get quite immediately from Theorem~\ref{t:ProSan} a) the following.

\begin{thm}\label{t:Godefroy}
Let $(M_n)$ be an increasing sequence of finite subsets of the metric space $M_{\ell_1}$ such that $M_{\ell_1}=\bigcup M_n$. 
Then for every $\varepsilon>0$, $D \in [1,2)$ and every $n\in \Natural$ there exists $k\in \Natural$ such that if $M_k \LipEmb{D} X$ then $\ell_1^n \subset X$ with linear distortion less than $1+\varepsilon$.
\end{thm}
The downside is that this theorem nor its proof do not provide any information about the dependence of $k$ on $n$.
The goal of Section~\ref{s:local} is dispensing of the ultraproduct argument in order to study the quantitative dependence between the size of the metric space $M_k$ and the dimension of $\ell_1^n$ (Theorem~\ref{t:main}). 

Finally in Section~\ref{s:ultra} we will prove Theorem~\ref{t:Godefroy} and indicate some directions for possible future research.

Throughout the paper we will need the following notions and notation. 
A mapping $f:M \to N$ between metric spaces $(M,d)$ and $(N,\rho)$ is called \emph{Lipschitz embedding} if there are constants $C_1,C_2>0$ such that
$C_1 d(x,y)\leq \rho(f(x),f(y))\leq C_2 d(x,y)$ for all $x,y \in M$. 
The distortion $\dist(f)$ of $f$ is defined as $\inf \frac{C_2}{C_1}$ where the infimum is taken over all constants $C_1,C_2$ which satisfy the above inequality. 
We say that $M$ Lipschitz embeds (embeds for brevity) into $N$ with distortion $D$ (in short $M \LipEmb{D} N$) if there exists a Lipschitz embedding $f:M\to N$ with $\dist(f)\leq D$. 
In this case, if the target space $N$ is a Banach space, we may always assume (by taking $C_1^{-1}f$) that $C_1=1$. 

If $X$ is a Banach space, we will denote by $\closedball{X}$ (resp. $\sphere{X}$) the closed unit ball (resp. unit sphere) of $X$.

For integers $m\leq n$ we denote $\llbracket m,n\rrbracket = [m,n]\cap \Natural$ and $\llbracket m,\infty \llbracket =[m,\infty)\cap \Natural$.
For a set $S$ and an integer $n \in \Natural$ we put $\nsub{S}{n}=\set{A \subset S: \abs{A}=n}$, the $n$-element subsets of $S$. 
If $x \in \Real$ we will denote by $\lceil x \rceil$ the smallest integer $n\geq x$.
\section{Low-distortion characterization of reflexivity}\label{s:reflexive}
Let $M_R=\set{\basepoint} \cup \Natural \cup F$ where $F=\set{\llbracket 1,n\rrbracket:n \in \Natural} \cup \set{\llbracket n,\infty \llbracket:n \in \Natural}$. 
We introduce on $M_R$ a graph structure: the edges are couples of the form $\set{\basepoint,a}$, $a\in \Natural$, or $\set{a, A}$, $a\in \Natural, A \in F$ and $a \in F$.
Finally, we equip $M_R$ with the shortest path distance.

\begin{thm}\label{t:reflexive}
a) Let $X$ be a Banach space and $D\in [1,2)$. If $M_R\LipEmb{D} X$ then $X$ is non-reflexive.

b) Conversely, if $X$ is non-reflexive then there is an equivalent norm $\abs{\cdot}$ on $X$ such that $M_R$ embeds isometrically into $(X,\abs{\cdot})$.
\end{thm}

\begin{proof}
a) Assume that $f:M \to X$ satisfies $f(\basepoint)=0$ and 
\[
d(x,y) \leq \norm{f(x)-f(y)}\leq Dd(x,y)
\]
for some $D<2$ and all $x,y \in M_R$.
Then, each $x^* \in \closedball{X}$ that norms $f(\llbracket 1,n\rrbracket)-f(\llbracket n+1,\infty \llbracket)$ satisfies
\[
\inf_{k\leq n} \duality{x^*,f(k)}-\sup_{k>n} \duality{x^*,f(k)} \geq 4-2D
\]
Hence $(f(n))_{n\in \Natural} \subset D\closedball{X}$ satisfies $\dist(\clco\set{f(i)}_{i=1}^n, \clco\set{f(i)}_{i=n+1}^\infty)\geq 4-2D$ for every $n \in \Natural$. 
By a well known lemma of James~\cite[page 51]{Beauzamy}, $X$ is not reflexive.

b) Let us first observe that $M_R$ embeds isometrically into $(c,\norm{\cdot}_\infty)$, the space of convergent sequences. 
We define $\Phi:M_R \to c$ by
\[
\begin{split}
\Phi(\basepoint)&=0\\
\Phi(\llbracket 1,n\rrbracket)&=2{\mathbf 1}_{\llbracket n+1,\infty\llbracket}\\
\Phi(\llbracket n,\infty \llbracket)&=-2{\mathbf 1}_{\llbracket 1,n\rrbracket}\\
\Phi(n)&=-{\mathbf 1}_{\llbracket 1,n\rrbracket} + {\mathbf 1}_{\llbracket n+1,\infty\llbracket}
\end{split}
\]
Then $\Phi$ is an isometric embedding. 
Indeed,
$\llbracket 1,n\rrbracket \cap \llbracket m,\infty \llbracket = \emptyset$ iff $m \geq n+1$. In this case the supports of $\Phi(\llbracket 1,n\rrbracket)$ and $\Phi(\llbracket m,\infty \llbracket)$ intersect and we have $\norm{\Phi(\llbracket 1,n\rrbracket)-\Phi(\llbracket m,\infty \llbracket)}_\infty=4$. 
Otherwise the supports do not intersect and we have $\norm{\Phi(I)-\Phi(J)}_\infty=2$.
For the possible distances between $\Phi(\llbracket 1,n\rrbracket)$ and $\Phi(m)$, resp. $\Phi(\llbracket n,\infty \llbracket)$ and $\Phi(m)$, consult Figure~\ref{f:c0}.

\begin{figure}[h]
\includegraphics{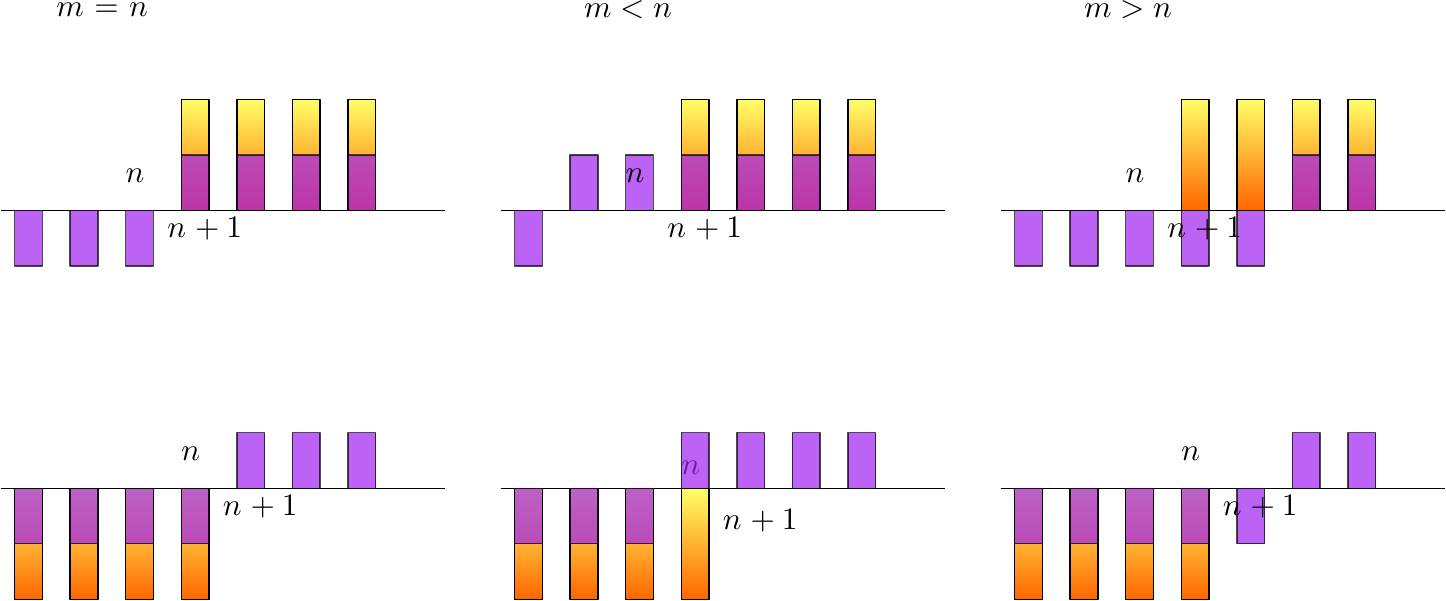} 
\caption{The purple bars correspond to $\Phi(m)$, the orange bars correspond to $\Phi(\llbracket 1,n\rrbracket)$ in the first line and to $\Phi(\llbracket n,\infty \llbracket)$ in the second line.}\label{f:c0}
\end{figure}

Now, let $Y$ be a one-codimensional subspace of $X$. 
Clearly, $Y$ is not reflexive.
Let $\theta \in (0,1)$.
By the proof of James lemma (see \cite[page 52]{Beauzamy}), there exist $F \in \closedball{Y^{**}}$ and sequences $(x_n) \subset \sphere{Y}$, $(x_n^*) \subset \sphere{Y^*}$ such that 
 \[
  \begin{split}
   F(x_n^*)=\theta &\mbox{ for all } n \in \Natural,\\
   x_n^*(x_k)=\theta &\mbox{ for all } n \leq k,\\
   x_n^*(x_k)=0 &\mbox{ for all } n>k.
  \end{split}
 \]
Observe that $X$ is isomorphic to $Z:=Y \oplus \mathrm{span}\set{F} \subset Y^{**}$. 
We are going to renorm $Z$ and embed $M$ isometrically into this renorming.

Let $y_n \in 2\closedball{Z}$ such that $x^*_n(y_n)=1$.
Let $C=\frac{\theta}{12}$.
We define, for every $n \in \Natural$ and every $x \in Z$,
\[
 \norm{x}_n:=\max\set{\abs{x^*_n(x)},C\norm{x-x^*_n(x)y_n}}.
\]
Then
\[
 \frac{C}{4} \norm{x} \leq \norm{x}_n\leq \norm{x}
\]
for every $n\in \Natural$ and so $\abs{x}:=\sup \norm{x}_n$ defines an equivalent norm on $Z$. 
We define an embedding $f:M \to (Z,\abs{\cdot})$ as follows
\[
\begin{split}
f(\basepoint)&=0\\
f(\llbracket 1,n\rrbracket)&=2(F-x_n)\\
f(\llbracket n,\infty \llbracket)&=-2x_n\\
f(n)&=F-2x_n.
\end{split}
\]
When evaluated against the functionals $(x_k^*)$ this embedding reminds the embedding $\Phi:M \to c$. 
Indeed, one can check easily that for every $a \in M$ and each $k \in \Natural$ we have $\duality{x_k^*,f(a)}=\theta\Phi(a)(k)$.
Moreover $C\norm{x-x^*_n(x)y_n}\leq \theta$ for all $n \in \Natural$ whenever $x=f(a)-f(b)$.
It follows that 
\[
\abs{f(a)-f(b)}=\sup_{n\in \Natural}\abs{x_n^*(f(a)-f(b))}=\theta\norm{\Phi(a)-\Phi(b)}_\infty=\theta d(a,b)
\] for every $a,b \in M$.
Thus $g=f/\theta$ is the desired embedding.
\end{proof}

\section{The role of stability}\label{s:Stability}
We recall that a metric space $(M,d)$ is \emph{stable} if for all bounded sequences $(x_n)$, $(y_n) \subset M$ and for all non-principal ultrafilters $\mathcal U, \mathcal V$ we have 
\[
\lim_{n,\mathcal U} \lim_{m,\mathcal V} d(x_n,y_m)= \lim_{m,\mathcal V} \lim_{n,\mathcal U} d(x_n,y_m).
\] 
The space $\ell_1$ is stable, see~\cite[page 212]{BLbook}.

The next proposition shows at once that in general we need to pass to a renorming in both theorems~\ref{t:ProSan}~b) and~\ref{t:reflexive}~b).

\begin{prop}
Let $C>1$. Let $M$ be uniformly discrete and bounded metric space with the property that $M \LipEmb{D} X$, $D<C$ only if $X$ is non-reflexive. Then $M \LipEmb{D} \ell_1 \Rightarrow D \geq C$.
\end{prop}

\begin{proof}
Let $\Phi:M \LipEmb{D} \ell_1, D<C$. 
Then $\Phi(M)$ is uniformly discrete, bounded and stable (see~\cite{BLbook}).
Baudier and Lancien~\cite{BaudLan} have shown that stable metric spaces nearly isometrically embed into the class of reflexive spaces.
In the uniformly discrete and bounded case that means that there exists a reflexive space $X$  such that $\Phi(M) \LipEmb{1} X$, thus $M \LipEmb{D} X$ which is impossible.
\end{proof}

\begin{rem}\label{r:BestDistortionInReflexive}
If $(x_n) \subset \closedball{X}$ is a $1$-separated sequence, then any bijection $f:M_R \to \set{x_n}$ satisfies $\dist(f)\leq 8$. 
Thus $M_R$ embeds with distortion $8$ into any infinite-dimensional Banach space.
We do not know whether $M_R$ embeds into some reflexive space $X$ with distortion $2$. 
In other words, we do not know whether the constant $2$ in Theorem~\ref{t:reflexive}~a) is optimal.

On the other hand, let $\rho:M_R\times M_R \to \Real$ be defined by
$\rho(\basepoint,n)=\frac23$, $\rho(\basepoint,A)=\frac43$, $\rho(n,A)=2$, $\rho(n,m)=\frac43$ and $\rho(A,B)=\frac83$ for all $n\neq m \in \Natural$ and all $A\neq B \in F \subset M_R$.
Then $\rho$ is a stable metric on $M_R$ and $\dist(Id)=3$ for the identity $Id:(M_R,d) \to (M_R,\rho)$.
To see the stability consider the following isometric embedding $g:(M_R,\rho) \to \ell_1(M_R)$ defined by $g(\basepoint)=0$, $g(n)= \frac23 e_n$, for every $n \in \Natural$, and $g(A)=\frac43 e_A$, for every $A \in F$.
It follows, using again~\cite{BaudLan}, that there is a reflexive Banach space $X$ such that $M_R \LipEmb{3} X$.

Finally let us mention that the method using~\cite{BaudLan} does not work for distortions less than $3$ as it follows from the next lemma that $M_R$ does not embed into a stable metric space with distortion less than $3$.
\end{rem}

\begin{lem}
Let $(M,d)$ be a metric space containing sequences $(x_n)$ and $(y_n)$ such that 
\[
\lim_n \lim_m d(x_n,y_m)\geq C \lim_m \lim_n d(x_n,y_m).
\] 
Then $M$ does not embed into any stable metric space with distortion  $D<C$.
\end{lem}

\begin{proof}
Let $D<C$ and assume that 
\[
 sd(x,y)\leq d(f(x),f(y))\leq s Dd(x,y)
\]
for some $f:M \to N$ and some stable metric space $N$.
Then 
\[
\begin{split}
 s \lim_n \lim_m  d(x_n,y_m)&\leq \lim_n \lim_m d(f(x_n),f(y_m))=\lim_m \lim_n d(f(x_n),f(y_m))\\ 
 & \leq sD \lim_m \lim_n d(x_n,y_m)<  sC\lim_m \lim_n d(x_n,y_m) \leq s\lim_n \lim_m d(x_n,y_m)
\end{split} 
\]
which is impossible.
\end{proof}

\begin{rem}\label{r:BanachMazur}
We recall the definition of the space $M_{\ell_1}$. 
The vertex set is $M_{\ell_1}=\set{\basepoint} \cup \Natural \cup F$ with $F=\set{A \subset \Natural: 1\leq \cardinality{A}<\infty}$. 
A pair $\set{a,b}$ is an edge either if $a=\basepoint$ and $b \in \Natural$, or if $a \in \Natural$, $b \in F$ and $a \in b$. 
The space $M_{\ell_1}$ is equipped with the shortest path metric.

In~\cite{ProSan}, we asked for the best constant $D$ such that $M_{\ell_1} \LipEmb{D} \ell_1$ and also for the value of $d_{BM}(\ell_1,\Free(M_{\ell_1}))$ where $\Free(M_{\ell_1})$ is the Lipschitz free space of $M_{\ell_1}$.

First, using the above lemma, one can easily see that $M_{\ell_1}$ does not embed with distortion less than $3$ into any stable space.
Second, if we define $\rho:M_{\ell_1}\times M_{\ell_1} \to \Real$ as in Remark~\ref{r:BestDistortionInReflexive}, we have that $\dist(Id)=3$ for the identity map $Id:(M_{\ell_1},d)\to (M_{\ell_1},\rho)$.
Defining the isometric embedding $g:(M_{\ell_1},\rho) \to \ell_1$ as in Remark~\ref{r:BestDistortionInReflexive}, this already gives that $D=3$.

But it also follows from the theory of Lipschitz free spaces that $\norm{\widehat{Id}}\norm{\widehat{Id}^{-1}}=3$ where $\widehat{Id}:\Free(M_{\ell_1},d),\Free(M_{\ell_1},\rho))$ is the unique linear extension of $Id$.
Moreover, $\Free(M_{\ell_1},\rho)\equiv \ell_1$. This can be seen by noticing that $\Free(M_{\ell_1},\rho)$ is isometric to a negligible subset of an $\Real$-tree which contains all the branching points, and applying~\cite{G}. 
Thus $d_{BM}(\ell_1,\Free(M_{\ell_1}))=3$.
\end{rem}

\section{Low-distortion representation of $\ell_1^n$}\label{s:local}

In this section we state and prove a quantitative version of Theorem~\ref{t:Godefroy} for a particular choice of spaces $M_n \subset M_{\ell_1}$.
Having in mind the definition of the space $M_{\ell_1}$ (see Remark~\ref{r:BanachMazur}),
the most natural choice of the spaces $M_n$ seems to be the following. 
We put $M_n=\set{\basepoint} \cup \llbracket 1,n \rrbracket \cup F_n$ where $F_n=2^{\llbracket 1,n \rrbracket}\setminus\set{\emptyset}$.
The graph structure and the metric are induced by the space~$M_{\ell_1}$.

The main result of this section follows.
\begin{thm}\label{t:main}
a) Let $D \in [1,\frac43)$ and $n \in \Natural$. 
Then $M_n \LipEmb{D} X$ implies that $\ell_1^n$ is $\displaystyle \frac{D}{4-3D}$-isomorphic to a subspace of $X$.\\
b) Let $D \in [1,2)$. 
For every $\alpha \in (0,1)$ there exists $\eta=\eta(\alpha,D)\in (0,1)$ such that $M_{k} \LipEmb{D} X$ implies that $\ell_1^{\lceil\eta k\rceil}$ is $\displaystyle \frac{2D}{2-D}$-isomorphic to a subspace of $X$ whenever $k > \displaystyle \frac{\log_2(\frac{2D}{2-D})}{1-\alpha}$.
\end{thm}

The isomorphism constant can be arbitrarily reduced, at the cost of augmenting the size of the metric space, by virtue of the following finite version of James's $\ell_1$-distortion theorem, see Proposition 30.5 in \cite{TJ}:
\emph{If $X$ contains a $b^2$-isomorphic copy of $\ell_1^{m^2}$, then $X$ contains a $b$-isomorphic copy of $\ell_1^m$.}

For example, in the case {\bf a)}, we get: \emph{If $D <\frac43$ and $w\geq -\log_2\displaystyle \left(\frac{\log_2(1+\varepsilon)}{\log_2(\frac{2D}{4-3D})}\right)$, then $M_{n^{2^w}} \LipEmb{D} X$ implies that $\ell_1^n$ is $(1+\varepsilon)$-isomorphic to a subspace of $X$.}

In order to prove Theorem~\ref{t:main} we will need the following lemma, which is a finite-dimensional version of Proposition~4 in \cite{Rosenthal}. 

\begin{lem}\label{l:basis}
Let $S$ be a set, $K>0$ and let $(f_i)_{i=1}^n \subset K\closedball{\ell_\infty(S)}$. 
Assume that there are $\delta>0$ and $r \in \Real$ such that these functions satisfy that for every $A \subset \llbracket 1,n \rrbracket$ there is $s \in S$ for which
\[\begin{split}
f_j(s)\leq r< r+\delta\leq f_i(s)
\end{split}
\]
for all $i \in A$ and for all $j \in \llbracket 1,n \rrbracket\setminus A$.
Then $(x_i)_{i=1}^n$ is $\frac{2K}{\delta}$-equivalent to the unit vector basis of~$\ell_1^n$.
\end{lem}
\begin{proof}
Let $(c_i)_{i=1}^n$ be a set of real coefficients. 
We put $A=\set{i\in \llbracket 1,n\rrbracket: c_i\geq 0}$. 
Let $s,t \in S$ satisfy
\[
 \begin{split}
  f_j(s)\leq r & <r+\delta\leq f_i(s)\\
  f_i(t)\leq r & <r+\delta\leq f_j(t)\\
 \end{split}
\]
for all $i \in A$ and $j \in \llbracket 1,n\rrbracket \setminus A$.
Then $f_i(s)-f_i(t) \geq r+\delta-r=\delta$ if $i \in A$ and $f_j(s)-f_j(t)\leq r-(r+\delta)=-\delta$ if $j \notin A$.
It follows that 
\[
 \sum_{1\leq i\leq n} c_i (f_i(s)-f_i(r))\geq \sum_{i \in A} c_i \delta - \sum_{i\notin A} c_i \delta=\delta\sum_{1\leq i \leq n}\abs{c_i}
\]
Hence $\max\set{\abs{\sum_{i=1}^n c_i f_i(s)},\abs{\sum_{i=1}^n c_i f_i(t)}} \geq \frac\delta2\sum\abs{c_i}$.
Thus
$
\frac\delta2\sum_{i=1}^n\abs{c_i} \leq \norm{\sum_{i=1}^n c_if_i}_\infty \leq K\sum_{i=1}^n\abs{c_i}.
$
\end{proof}

\begin{proof}[Proof of Theorem~\ref{t:main} a)]
Assume that $f:M_n \to X$ satisfies $f(\basepoint)=0$ and 
\[
d(x,y) \leq \norm{f(x)-f(y)}\leq Dd(x,y)
\]
for some $D<\frac43$ and all $x,y \in M_n$.
We put 
\[
\displaystyle X_{a,b}=\set{x^* \in \closedball{X^*}:\duality{x^*,f(a)}\geq 4-3D, \duality{x^*,f(b)}\leq -4+3D}.
\]
We claim that for every $A,B \subset \llbracket 1,n\rrbracket$ such that $A \cap B = \emptyset$ and $A \cup B \neq \emptyset$, we have 
\[
\bigcap_{a \in A,b\in B} X_{a,b} \neq \emptyset.
\]
Indeed, in the case $A \neq \emptyset \neq B$ we take $x^* \in K$ such that $\duality{x^*,f(A)-f(B)}=\norm{f(A)-f(B)}\geq 4$.
Then for any $a \in A$ we have
\[
\begin{split}
\duality{x^*,f(a)-f(\basepoint)}&=\duality{x^*,f(A)-f(B)}-\duality{x^*,f(A)-f(a)}-\duality{x^*,f(\basepoint)-f(B)}\\
&\geq 4-3D,
\end{split}
\]
and by the same argument we get $\duality{x^*,f(b)}\leq -4+3D$. 
Hence $\displaystyle x^* \in \bigcap_{a \in A,b\in B} X_{a,b}$.
In the case when  $B=\emptyset$ we get an $x^* \in \closedball{X^*}$ such that for
all $a \in A$ we have $\duality{x^*,f(b)}\geq 2-D \geq 4-3D$.
Lemma~\ref{l:basis} now implies that $(f(i))_{i=1}^n$ is $\frac{D}{4-3D}$-equivalent to the unit vector basis of $\ell_1^n$.
\end{proof}

The proof of Theorem~\ref{t:main} b) will depend on the following well known lemma.

\begin{lem}\label{l:KeyPajor}
For every $\alpha \in (0,1]$ there is $\eta=\eta(\alpha)>0$ such that for each $k\in \Natural$ and each $\mathcal S \subset 2^{\llbracket1,k\rrbracket}$, the estimate $\cardinality{\mathcal S} \geq 2^{\alpha k}$ implies that there exists $H \in \nsub{\llbracket1,k\rrbracket}{\lceil\eta k\rceil}$ such that $\set{A \cap H: A \in \mathcal S}=2^H$.
\end{lem}

The proof of this lemma is in turn based on the combination of the following two lemmas, see \cite[pages 402-403]{LQ} or \cite[page 11]{Paj}, 
\begin{lem}[Sauer, Shelah, and Vapnik and \v Cervonenkis]\label{l:Sauer}
 Let $\mathcal S \subset 2^{\llbracket1,k\rrbracket}$ such that $\cardinality{\mathcal S} > \sum_{i=0}^{m-1} {k \choose i}$ for some $m\leq k$. Then there is $H \in \nsub{\llbracket1,k\rrbracket}{m}$ such that $\set{A \cap H: A \in \mathcal S}=2^H$.
\end{lem}

\begin{lem}\label{l:BinomialSumEstimate}
For every $1\leq m \leq k$ one has
\[
 \sum_{i=0}^m {k \choose i}\leq \left(\frac{ek}{m}\right)^m.
\]
\end{lem}

\begin{proof}[Proof of Lemma~\ref{l:KeyPajor}]
 Let $\eta>0$ satisfy $2^\alpha>\left(\frac{e}{\eta}\right)^\eta$. And $m=\lceil \eta k\rceil$. Then, using $\eta<1$ and Lemma~\ref{l:BinomialSumEstimate}, 
\[
 \cardinality{\mathcal S}\geq 2^{\alpha k}>\left(\frac{ek}{\eta k}\right)^{\eta k} \geq \left(\frac{ek}{m-1}\right)^{m-1} \geq \sum_{i=0}^{m-1} {k \choose i}
\]
So we get the existence of $H$ by Lemma~\ref{l:Sauer}.
\end{proof}

\begin{proof}[Proof of Theorem~\ref{t:main} b)]
Given $D<2$, let $c=\lceil\frac{2D}{2-D}\rceil-1$. 
Let $k>\frac{\log_2(c-1)}{1-\alpha}$. 
Assume that $f:M_k \to X$ satisfies $f(\basepoint)=0$ and 
\[
d(x,y) \leq \norm{f(x)-f(y)}\leq Dd(x,y)
\]
for all $x,y \in M_k$.
To simplify notation we will denote $x':=f(x)$ for each $x\in M_k$. 
Let $r_j:=-D+j(D-2)$ for $j \in \llbracket 1,c\rrbracket$. 
Then any closed interval of lenght $4-2D$ which is contained in $[-D,D]$ contains at least two different points $r_j$.
We claim that for every $A \in 2^{\llbracket 1,k\rrbracket}$ there are $j \in \llbracket 1,c-1\rrbracket$ and $x^* \in \closedball{X^*}$ such that 
\[
\duality{x^*,b'}\leq r_j<r_{j+1}\leq \duality{x^*,a'}
\] 
for all $a \in A$ and all $b\in B:= \llbracket1,k\rrbracket\setminus A$. 
Indeed, we take $x^* \in \closedball{X^*}$ such that 
\[
\begin{split}
\duality{x^*,A'-B'}&=\norm{A'-B'}\mbox{ if } A \neq \emptyset\neq B\\
\duality{x^*,A'}&=\norm{A'}\mbox{ if } B=\emptyset\\
\duality{x^*,-B'}&=\norm{B'}\mbox{ if } A=\emptyset
\end{split}
\]
In the first case we get for all $a \in A$ and all $b\in B$ that $\duality{x^*,a'-b'} \geq 4-2D$. 
Moreover $\duality{x^*,a'},\duality{x^*,b'} \in [-D,D]$ so the claim follows by the choice of $(r_j)$.
In the second case we get for all $a \in A$ that $\duality{x^*,a'}\geq 2-D$, and in the last case we get for all $b\in B$ that $\duality{x^*,b'}\leq D-2$. 
In both cases we use that the interval $[D-2,2-D]$ contains at least two different $r_j$ to finish the proof of the claim.

Let us choose for every $A \subset \llbracket 1,k\rrbracket$ one such $j$ which we will denote $j_A$.
By the pigeonhole principle, there is some $j \in \llbracket1,c-1\rrbracket$ such that $\cardinality{\mathcal S}\geq \frac{2^k}{c-1}$ for $\mathcal S=\set{A \in 2^{\llbracket1,k\rrbracket}:j_A=j}$. 
By the choice of $k$ we have $\frac{2^k}{c-1} \geq 2^{\alpha k}$. 
Let $\eta\in (0,1)$ and $H \in \nsub{\llbracket1,k\rrbracket}{\lceil\eta n\rceil}$ be as in Lemma~\ref{l:KeyPajor}. 
Applying Lemma~\ref{l:basis} we can see that $(f(i))_{i\in H}$ is $\frac{2D}{2-D}$-equivalent to the unit vector basis of $\ell_1^{\lceil \eta k\rceil}$.
\end{proof}

\section{Ultraproduct techniques for low distortion representation}\label{s:ultra}
Let us first see the proof of Theorem~\ref{t:Godefroy}.
\begin{proof}[Proof of Theorem~\ref{t:Godefroy}]
Suppose that the assertion is not true for some $\varepsilon>0$, $D \in [1,2)$ and $n \in \Natural$. Then for every $k$ there is $X_k$ such that $M_k \LipEmb{D} X_k$ and $\ell_1^n$ is not $(1+\varepsilon)$-isomorphic to a subspace of $X_k$. Let $X=\prod X_k$ be an ultraproduct along some free ultrafilter on $\Natural$. Then $M_{\ell_1} \LipEmb{D} X$ and so, by Theorem~\ref{t:ProSan}~a), $\ell_1$ embeds into $X$ linearly and does so arbitrarily well (by James's $\ell_1$-distortion theorem, see e.g.~\cite{AlbiacKalton}) . Therefore $\ell_1^n$ embeds into $X$ arbitrarily well and in particular it must be $(1+\frac\varepsilon2)$-isomorphic to a subspace of some $X_k$. Contradiction.
\end{proof}

Let us mention the following folklore result which according to G. Lancien goes back to G.~Schechtmann.
\begin{thm}
Let $X$ be a Banach space such that $\dim X<\infty$. Then for every $D\geq 1$ and $\varepsilon>0$ there is a finite set $F \subset X$ such that for any given Banach space $Y$ the fact $F \LipEmb{D} Y$ implies that $X$ is $(D+\varepsilon)$-isomorphic to a subspace of $Y$.
\end{thm}
\begin{proof}
Let $(F_n) \subset 2^X$ be any increasing family of finite sets such that $\overline{\bigcup F_n}=X$.
Let us assume that there are $D>1$ and $\varepsilon>0$ such that for every $k \in \Natural$ there is a Banach space $Y_k$ and an embedding $f_k:F_k \LipEmb{D} Y_k$ but $X$ is not $(D+\varepsilon)$-isomorphic to any subspace of $Y_k$. 
Then $\Phi(x):=\left[(f_n(x))_n\right]$ for $x \in \bigcup F_n$ and extended by continuity to the whole $X$ is an embedding of $X$ into $\prod_{\mathcal U} Y_k$ with distortion $D$.
By the theorem of Heinrich and Mankiewicz \cite[Theorem 7.9]{BLbook} $X$ linearly embeds into $(\prod_{\mathcal U} Y_k)^*$ with distortion $D$. By local reflexivity $X$ embeds linearly into $\prod_{\mathcal U} Y_k$ with distortion $D+\varepsilon/3$ and therefore $X$ embeds into some $Y_k$ linearly with distortion $D+2\varepsilon/3$ which is a contradiction. 
\end{proof}

\begin{rem}
1) Clearly the above theorem is qualitatively better than Theorem~\ref{t:Godefroy} as it works for any space $X$ and withouth any restriction on the distortion. 
Also the set $F$ is isometrically in~$X$. 
On the other hand, any other information on the nature of $F$ is completely inaccessible.
It could be interesting to give a concrete example of such sets for a given finite dimensional space~$X$.

2) When $X=\ell_p^n$ with $1\leq p \leq 2$, one can say more.
Let $C_p^n$ be the ``$n$-cube'' equipped with the $\ell_p$-norm. We recall a result of Bourgain, Milman and Wolffson \cite[page 297]{BMW} which says: Let $1\leq p \leq 2$. Assume that there exists $D>0$ such that $C_p^n \LipEmb{D} Y$ for every $n$. Then for every $\varepsilon>0$ and $n\in \Natural$ there is a subspace of $Y$ which is $(1+\varepsilon)$-isomorphic to $\ell_p^n$.

Essentially the same proof as the one of Theorem~\ref{t:Godefroy} gives therefore the following.
Let $1\leq p \leq 2$. 
Then for every $\varepsilon>0$, $D>0$ and $n \in \Natural$ there is $k \in \Natural$ such that for every Banach space $Y$ we have that
$C_p^k \LipEmb{D} Y$ implies that $\ell_p^n$ is $(1+\varepsilon)$-isomorphic to a subspace of $Y$.

The dependence of $k$ on $n$ does not seem to follow from the proof in~\cite{BMW}.
\end{rem}

\section*{Acknowledgements}
I'm grateful to Gilles Godefroy for pointing out Theorem~\ref{t:Godefroy} to me. 
I am also grateful to Beata Randrianantoanina and Gilles Lancien for inspiring discussions. 
A part of this paper was written during the conference ``Banach spaces and their applications in analysis'' at CIRM. 
I would like to thank the organizers for inviting me, and the CIRM for its hospitality.

\end{document}